\documentclass[12pt,reqno]{amsart}

\advance \topmargin by -\headheight
\advance \topmargin by -\headsep
\evensidemargin \oddsidemargin
\usepackage{amsmath}
\usepackage{bbm}
\usepackage{amsfonts}
\usepackage{amssymb}
\usepackage{amsthm}
\usepackage{csquotes}
\usepackage{color}
\usepackage{mathrsfs}
\usepackage{dsfont}
\usepackage{bm}
\usepackage{cite}
\usepackage{enumerate}
\usepackage{booktabs}
\usepackage[mathscr]{eucal}
\usepackage{mathtools}
\usepackage{CJKutf8}

\numberwithin{equation}{section}

\newtheorem{theorem}{Theorem}[section]
\newtheorem{corollary}[theorem]{Corollary}
\newtheorem{lemma}[theorem]{Lemma}
\newtheorem{proposition}[theorem]{Proposition}
\newtheorem{fact}[theorem]{Fact}
\theoremstyle{definition}

\newtheorem{example}[theorem]{Example}

\newcommand{\R}{\mathbb R}
\newcommand{\T}{\mathbb T}
\newcommand{\Z}{\mathbb Z}

\newcommand{\one}{\mathbbm 1}
\newcommand{\dd}{\,\mathrm d}

\title[Freiman's $3k-4$ theorem in connected abelian groups]{Geometric lifting and Freiman's $3k-4$ theorem in compact connected abelian groups}
\author{Yifan Jing}
\address{Department of Mathematics, The Ohio State University, Columbus OH, 43210, USA}
\email{jing.245@osu.edu}
\author{Yuchen Meng}
\address{Department of Mathematics, The Ohio State University, Columbus OH, 43210, USA}
\email{meng.730@osu.edu}

\begin{document}

\begin{abstract}
We develop a geometric lifting method for inverse sumset problems in compact connected abelian groups. The first result is an analogue of Freiman's $3k-4$ theorem, that is every compact set $A\subseteq G$ of sufficiently small Haar measure satisfying $\mu_G(A+A)<3\mu_G(A)$ is contained in a one dimensional Bohr set of measure at most $\mu_G(A+A)-\mu_G(A)$. This resolves a question of Christ and Iliopoulou. The proof combines Bilu's theorem with a geometric refinement of the spillover argument.  We also establish a sharp projection theorem. Under a continuous surjective homomorphism with connected kernel, a compact set of sufficiently small positive measure and doubling at most $K$, where $2\le K<3$, has image of doubling at most $2K-2$, and this factor is best possible. Further consequences include variants of the $3k-4$ theorem for popular sumsets and an inverse theorem for Tao's convolution inequality.
\end{abstract}

\maketitle

\section{Introduction}\label{sec:introduction}
The purpose of this paper is to develop a geometric lifting method for sets of small doubling in compact connected abelian groups. For a compact group $G$, we write $\mu_G$ for its normalized Haar measure. We write $\T=\mathbb{R}/\mathbb{Z}$ and equip it with normalized Lebesgue measure. A one dimensional Bohr set in $G$ is a set of the form $\chi^{-1}(I)$, where $\chi:G\to\T$ is a continuous surjective homomorphism and $I\subseteq\T$ is a compact arc. 

\subsection{Freiman's $3k-4$ theorem}

Let $G$ be a compact connected abelian group. Kneser's theorem~\cite{Kneser} asserts that for any nonempty compact $A\subseteq G$ one has
\[
    \mu_G(A+A)\geq \min\{1,2\mu_G(A)\}.
\]
This can be viewed as a continuous generalization of the Cauchy--Davenport theorem. Kneser also studied the equality case, and equality occurs when $A$ is a one-dimensional Bohr set. Such sets are the continuous counterparts of the one dimensional arithmetic progressions in Vosper's inverse Cauchy--Davenport theorem.

The structure of sets for which equality in Kneser's theorem nearly occurs was clarified by Tao~\cite{TaoKneser}. With the density assumptions, his inverse theorem shows that sets satisfying $\mu_G(A+A)\leq (2+o(1))\mu_G(A)$ are close in measure to one dimensional Bohr sets. Christ and Iliopoulou~\cite{CI22} obtained a quantitative containment theorem for sets with $\mu_G(A+A)\leq (2+\eta)\mu_G(A)$ with a small positive absolute constant $\eta>0$. A different proof of Christ--Iliopoulou theorem via a spillover argument is later given by the first author and Tran in \cite{JingTranKIP}, where they prove a more general result for nonabelian groups as well. The abelian case is refined by the first author~\cite{J} and he obtained an inverse theorem with the same one dimensional Bohr set conclusion for $\eta$ is roughly $0.3$. 

Thus one has a conclusion for doubling at most $2+\eta$, with the permitted $\eta>0$. On the other hand, two-dimensional structure can occur beyond the one-dimensional range, with natural endpoint is $3$ (or $\eta =1$), as in the classical Freiman $3k-4$ theorem over $\Z$. See Example~\ref{eq: 3} for details. 

In the same paper, Christ and Iliopoulou suggested that the precise inverse results for the circle should extend to arbitrary compact connected abelian groups for sets with small measures. We confirm their conjecture.  The following is an asymmetric version for general $G$. In particular, $G$ need not be metrizable.

\begin{theorem}\label{thm: main Freiman}
For every $\tau\geq1$ there is a constant $\eta_\tau>0$ with the following property. Let $G$ be a compact connected abelian group and let $A,B\subseteq G$ be compact sets satisfying
\begin{equation*}
 0<\tau^{-1}\mu_G(A)\leq\mu_G(B)\leq\mu_G(A)\leq\eta_\tau.
\end{equation*}
Suppose
\begin{equation*}
 \mu_G(A+B)<\mu_G(A)+2\mu_G(B),
\end{equation*}
then there exist a continuous surjective homomorphism $\chi:G\to\T$ and compact arcs $I,J\subseteq\T$ such that $A\subseteq \chi^{-1}(I)$, $B\subseteq \chi^{-1}(J)$, and
\begin{equation*}
\mu_\T(I)\leq \mu_G(A+B)-\mu_G(B),\qquad \mu_\T(J)\leq \mu_G(A+B)-\mu_G(A). 
\end{equation*}
\end{theorem}

Taking $B=A$ and $\tau=1$ gives the compact connected $3k-4$ theorem in its usual form: every compact set $A$ of sufficiently small positive measure satisfying $\mu_G(A+A)<3\mu_G(A)$ is contained in a one dimensional Bohr set of measure at most $\mu_G(A+A)-\mu_G(A)$. In the usual $3k-4$ theorem over $\Z$, the set $A$ is contained in a one-dimensional arithmetic progression $P$ with $|P|\leq |A+A|-|A|+1$. Here the arithmetic progression is replaced by a one-dimensional Bohr set. The endpoint correction $+1$ disappears as the endpoints of an arc have measure $0$ preimages under a surjective character. 

Kneser's theorem also covers the disconnected case. The structure of pairs satisfying $\mu_G(A+B)<\mu_G(A)+\mu_G(B)$ in this setting was described by Kemperman~\cite{KempermanKneser}, and an inverse theorem of $(2+\eta)$-type was proved by Griesmer~\cite{Griesmer}. There is also a nonabelian analogue of Kneser's theorem for locally compact groups, due to Kemperman~\cite{Kemperman}. For connected groups, the equality case and an inverse theorem of $(2+\eta)$-type were established by the first author and Tran~\cite{JingTranKIP}. 

\subsection{The geometric lifting argument}

Our proof combines Bilu's $3k-4$ theorem for finite dimensional tori~\cite{Bilu} with a geometric refinement of the spillover and projection-and-lifting argument used by the first author and Tran in~\cite[Section~11]{JingTranKIP}, which is different from the earlier approaches by Tao and by Christ--Iliopoulou. Roughly speaking, spillover argument provides a way to control the doubling in the projection. Let $\pi: G\to Q$ be a group homomorphism, and one may hope that $\mu_G(A+A)<3\mu_G(A)$ would imply $\mu_Q(\pi A+\pi A)<3\mu_Q(\pi A)$ and one may then apply induction on dimension to control $\pi(A)$. This approach has two issues. First, one can construct sets with doubling less than $3$ in $G$ but doubling close to $4$ in the projection, even when all of $G,H,\ker\pi$ are connected. We will discuss this phenomenon in details in our next main theorem. Second, even if the doubling of $\pi(A)$ is less than $3$, $\mu_Q(\pi A + \pi A) - \mu_Q(\pi A)$ can be way larger than our target value $\mu_G(A+A)-\mu_G(A)$. Indeed, this is the approach used in \cite{J} where the first author refines the spillover argument so that when the doubling of $A$ is less than $2.3$, the doubling of the projection is at most $3$ and one can also control  $\mu_Q(\pi A + \pi A) - \mu_Q(\pi A)$ by $\mu_G(A)$. The above two issues show the limitation of this approach.

The new ingredient in this work is a two step projection estimate and a fine control on the lifting step using the metric on the quotient. More precisely, we consider $\pi_1: G\to \T^d$ and $\pi_2: \T^d\to\T$. The first homomorphism is given by Gleason--Yamabe type result, and as the kernel can be made in an arbitrarily small open neighborhood, one can control the doubling in $\T^d$, which fixes the first issue mentioned in the previous paragraph. Bilu's theorem in $\T^d$ then provides a character $\pi_2$ to push down to $\T$. Two issues both occur again for $\pi_2$: the doubling in $\T$ can be too large and  $\mu_\T(\pi_2 A + \pi_2 A) - \mu_\T(\pi_2 A)$ is not controlled. The final piece is to use metric in $\T$. Bilu's theorem in $\T^d$ ensure that the projection of sets in $\T$ under $\pi_2$ still behaves like one-dimensional and unwrapping, though the doubling can be large. Hence one can lift $G$ and $\T$ to their universal covers and use geometry in $\R$ to control the length of the interval (more precisely, use the order in $\R$ to bound the end points of the sumset). Similar lifting idea is also used by the first author and Mudgal~\cite[Lemma 3.1]{JM}. This is presented in Section 3.

It is worth noting that the sufficiently small measure restriction in Theorem~\ref{thm: main Freiman} comes from Bilu's theorem in $\T^d$. Our method lifts the quotient result back to $G$, so the density restriction in the finite dimensional theorem is preserved. When $A=B$, once the initial Bohr containment has been obtained, the lifting argument itself only requires $\mu_G(A)<1/4$. We therefore expect that the conclusion of our $3k-4$ theorem should hold whenever $\mu_G(A)<1/4$ and $\mu_G(A+A)<3\mu_G(A)$. On the other hand, the restriction at $1/4$ is unavoidable; see the example in~\cite{CandelaDeRoton}.

\subsection{Small doubling under projection}

A second consequence of the method is a sharp bound on the doubling of the image of a small doubling set under projection. Let
\[
    0\to H\to G\xrightarrow{\pi}Q\to0
\]
be a short exact sequence of compact abelian groups, and let $A\subseteq G$ be compact with positive measure. It is well known that $\mu_G(A+A)/\mu_G(A)\leq K$ does not imply that the doubling of $\pi(A)$ is also at most $K$. Kong, Peng, and Tran~\cite{KPT} proved that
\[
    \frac{\mu_Q(\pi(A)+\pi(A))}{\mu_Q(\pi(A))}\leq K^2,
\]
and showed that this bound is sharp in general.

On the other hand, when the exact sequence is a sequence of compact connected abelian groups, one may expect a better control on the doubling of the projection set. Indeed, using Kneser's theorem~\cite{Kneser}, one can show that $\mu_G(A+A)/\mu_G(A)\leq 2$ implies $\mu_Q(\pi(A)+\pi(A))/\mu_Q(\pi(A))\leq 2$. Our first theorem proves a sharp projection theorem for doubling less than $3$ in compact connected abelian groups. 

\begin{theorem}
\label{thm:projection sym}
There is a constant $\eta>0$ with the following property. Let $2\leq K<3$, and let $0\to H\to G\xrightarrow{\pi} Q\to 0$ be an exact sequence of compact connected abelian groups. Suppose $A\subseteq G$ is compact and $0<\mu_G(A)<\eta$, and
\[
\mu_G(A+A)\leq K\mu_G(A)
\]
Then either $\pi(A)+\pi(A) = Q$, or we have
\[
\mu_Q(\pi(A)+\pi(A))\leq2(\mu_G(A+A)-\mu_G(A))\leq (2K-2)\mu_Q(\pi(A)). 
\]
\end{theorem}

We also prove an asymmetric version involving two sets $A$ and $B$, see Theorem~\ref{thm:projection}. In Section~5, we construct examples showing that the factor $2K-2$ in Theorem~\ref{thm:projection sym} is sharp for every $2\leq K<3$. We also give examples with disconnected kernel for which the estimate in Theorem~\ref{thm:projection sym} fails. Thus the connectedness of $H$ is necessary for the conclusion. This should be distinguished from the geometric lifting estimate in
Section~3, which allows disconnected character kernels.

\subsection{Popular sumsets and Tao's convolution inequality}

We present two consequences of Theorem~\ref{thm: main Freiman}. The first concerns popular sumsets. For $t>0$, let
\[
\mathscr{S}_t(A,B) =\{ x\in A+B : \one_A*\one_B (x)\geq t\}
\]
be the popular sumset of $A$ and $B$, and we write $S_t(A,B)=\mu_G(\mathscr{S}_t(A,B))$. 

\begin{theorem}
\label{thm: popular Freiman}
For every $\tau\geq1$ there is $\eta_\tau>0$ with the following property. For every $0<\alpha\leq\eta_\tau$ and $0<\kappa,\varepsilon<1$, there exists $t_0=t_0(\tau,\alpha,\kappa,\varepsilon)>0$ such that the following holds. Let $G$ be connected compact abelian group, and let $A,B\subseteq G$ be compact with
\[
 \alpha\leq\mu_{G}(A)\leq\eta_\tau,
 \qquad \tau^{-1}\mu_{G}(A)\leq \mu_{G}(B)\leq \mu_{G}(A).
\]
If $0<t\leq t_0$ and
\[
 S_t(A,B)\leq \mu_{G}(A)+(2-\kappa)\mu_{G}(B),
\]
then there exist a continuous surjective homomorphism $\chi:G\to\T$ and compact arcs $I,J\subseteq\T$ such that $\mu_{G}(A\setminus\chi^{-1}(I))\leq\varepsilon \mu_{G}(B)$, $\mu_{G}(B\setminus\chi^{-1}(J))\leq\varepsilon \mu_{G}(B)$, and
\[
 \mu_\T(I)\leq S_t(A,B)-\mu_{G}(B)+\varepsilon \mu_{G}(B),\qquad \mu_\T(J)\leq S_t(A,B)-\mu_{G}(A)+\varepsilon \mu_{G}(B).
\]
\end{theorem}

In \cite[Theorem 1.1]{TaoKneser}, Tao proved a generalization of Kneser's inequality, that for $A,B\subseteq G$ are compact and $G$ is compact connected abelian,
\[
\frac{1}{t}\int_G \min\{\one_A*\one_B(x),t\} \dd\mu_G(x)\geq \min\{\mu_G(A)+\mu_G(B)-t,1\}. 
\]
Here we obtained a strong inverse theorem of the inequality of $3k-4$ type.  

\begin{theorem}\label{thm: pop truncated}
With the same density and parameters assumptions of Theorem~\ref{thm: popular Freiman}, if
\[
 \frac{1}{t}\int_G \min\{\one_A*\one_B(x),t\} \dd\mu_G(x)\leq \mu_G(A)+(2-\kappa)\mu_G(B),
\]
then there are a continuous surjective homomorphism $\chi:G\to\T$ and compact arcs $I,J$ satisfying $\mu_{G}(A\setminus\chi^{-1}(I))\leq\varepsilon \mu_{G}(B)$, $\mu_{G}(B\setminus\chi^{-1}(J))\leq\varepsilon \mu_{G}(B)$, and
\begin{align*}
    \mu_\T(I)\leq  \frac{1}{t}\int_G \min\{\one_A*\one_B(x),t\} \dd\mu_G(x) - \mu_G(B) +\varepsilon\mu_G(B),\\
    \mu_\T(J)\leq  \frac{1}{t}\int_G \min\{\one_A*\one_B(x),t\} \dd\mu_G(x) - \mu_G(A) +\varepsilon\mu_G(B). 
\end{align*}
\end{theorem}

This answers a question suggested by Tao~\cite{TaoKneser}. A weaker version of the result, that is when $2-\kappa$ being replaced by $1+\eta$,  can be derived from Christ--Iliopoulou.

\subsection*{Organization}
Section~\ref{sec:preliminaries} records basic results we will use from abstract harmonic analysis, including the quotient integral formula and properties of Haar measures. We also record the additive combinatorics input, that is Bilu's theorem. A refined projection theorem is proved in Section 3. In Section 4, we prove the Freiman $3k-4$ Theorem. In Section 5 we proved the sharp projection theorem for small doubling sets. Section 6 records applications to popular sumsets and inverse Tao's inequality.

\subsection*{Statement of AI}
The paper is written by the authors. ChatGPT is used for checking the errors when the first version is written, polishing the abstract, and help the authors to restructure the introduction. 

\subsection*{Acknowledgments} 
The authors thank Akshat Mudgal for discussions about inverse Tao's inequality, which motivated the authors to write Section 6. 
The first author is supported by NSF Grant DMS-2503063. 

\section{Preliminaries}\label{sec:preliminaries}

\subsection{Topological groups}
To make the paper more accessible to readers without a background in locally compact groups, we include here a brief introduction of the basic properties we used about the Haar measures. 
We say that a measure $\mu$ on a $\sigma$-algebra of Borel subsets of a locally compact abelian group $G$ is a \emph{Haar measure} on $G$ if the following conditions hold:
\begin{enumerate}
     \item (translation-invariance)  $\mu(X) =\mu(a+X)$ for all $a\in G$ and all measurable sets $X \subseteq G$.
    \item (inner and outer regular) When $X$ is open, $\mu(X) =\sup \mu(K)$ with $K$ ranging over compact subsets of $X$.  When $X$ is Borel,  $\mu(X) = \inf \mu(U)$ when $U$ ranging over open subsets of $G$ containing $X$.
    \item (compactly finite) $\mu$ takes finite measure on compact subsets of $G$.
    \item (measurability characterization) If there is an increasing sequence $(K_n)$ of compact subsets of $X$, and a decreasing sequence $(U_n)$ of open subsets of $G$ with $X \subseteq U_n$ for all $n$ such that $\lim_{n \to \infty} \mu(K_n) = \lim_{n \to \infty} \mu(U_n) $, then $X$ is measurable.
\end{enumerate}
When $G$ is compact, we use $\mu_G$ to denote the \emph{normalized Haar measure} on $G$, that is, $\mu_G(G)=1$. 

We use the following integral formula ~\cite[Theorem 2.49]{Folland} in our proofs.
\begin{fact}[Quotient integral formula]
Let $G$ be a locally compact group, and let $H$ be a closed normal subgroup of $G$. Let $\mu_G$ and $\mu_H$ be left Haar measures on $G$ and on $H$, respectively. Then  there is a unique left Haar measure $\mu_{G/H}$ on $G/H$, such that
for every $f\in C_c(G)$, 
\[
\int_G f(x)\dd \mu_G(x)=\int_{G/H}\int_H f(xh) \dd \mu_H(h)\dd \mu_{G/H}(xH).
\]
\end{fact}

As a directly corollary, let $\pi: G\to Q$ be a homomorphism with compact kernel, then $\mu_{Q}(\pi(A))\geq \mu_G(A)$ for any compact $A\subseteq G$. 

We also repeatedly use the following connectedness principle. 
\begin{lemma}\label{lem: connected}
Suppose $\pi:G\to Q$ is a continuous surjective homomorphism of topological groups.  If $Q$ and $\ker \pi$ are connected, then $G$ is connected.
\end{lemma}

\begin{proof}
If $U\subseteq G$ is clopen, then its intersection with every coset of $\ker \pi$ is clopen in that connected coset. Hence each fiber lies either in $U$ or in $G\setminus U$.  Since $\pi$ is open, $\pi(U)$ and $\pi(G\setminus U)$ are disjoint open subsets whose union is $Q$. Connectedness of $Q$ makes one of them empty.
\end{proof}

We should also note that connectedness of $G$ and $Q$ does not force $\ker\pi$ to be connected.

\subsection{Additive combinatorics}

We use the connected form of Kneser's inequality. 
\begin{fact}[Kneser's inequality]
Let $G$ be a compact connected abelian group, and $A,B\subseteq G$ are compact sets. Then
\[
\mu_G(A+B)\geq \min\{\mu_G(A)+\mu_G(B),1\}. 
\]
\end{fact}

The inverse theorem input of the paper is Bilu's theorem~\cite{Bilu}, where he proved a Freiman's $3k-4$ theorem over tours. 
\begin{fact}[Bilu's theorem]\label{fact:toral}
For every $\tau\geq1$ there is a constant $\theta_\tau>0$, independent of $n$, with the following property.  If $A,B\subseteq\T^n$ are compact and
\[
 0<\tau^{-1}\mu_{\T^n}(A)
 \leq\mu_{\T^n}(B)
 \leq\mu_{\T^n}(A)
 \leq\theta_\tau,
\]
then either
\[
 \mu_{\T^n}(A+B)
 \geq\mu_{\T^n}(A)+2\mu_{\T^n}(B),
\]
or there exist a continuous surjective character $\psi:\T^n\to\T$ and compact arcs $I,J\subseteq\T$ such that
\begin{equation}\label{eq:toral-carrier}
 \begin{aligned}
 X&\subseteq\psi^{-1}(I),
 &\mu_\T(I)&=\mu_{\T^n}(X+Y)-\mu_{\T^n}(Y),\\
 Y&\subseteq\psi^{-1}(J),
 &\mu_\T(J)&=\mu_{\T^n}(X+Y)-\mu_{\T^n}(X).
 \end{aligned}
\end{equation}
\end{fact}

This is known as $\alpha+2\beta$ inequality in the literature. For a sharper theorem specific to the circle, see Candela and de Roton~\cite{CandelaDeRoton}.

Finally, we include a classic example in $\T$, showing that $\mu_G(A+A)<3\mu_G(A)$ in Theorem~\ref{thm: main Freiman} cannot be improved  to $\mu_G(A+A)\leq 3\mu_G(A)$. 

\begin{example}\label{eq: 3}
Given $0<\delta<1/6$. Let $A= [0,\delta]\cup \{1/3\}$ be a compact subset of $\T$. Then $\mu_\T(A) = \delta$, and 
\[
A+A = [0,2\delta] \cup [1/3, 1/3+\delta]\cup \{2/3\}. 
\]
In particular, $\mu_\T(A+A) = 3\delta = 3\mu_\T(A)$. 

The smallest compact interval contains $A$ is $[0,1/3]$ with length $>2\delta$. In general, any surjective homomorphism mapping $\T$ to itself is $x\mapsto nx$ where $n$ is a positive integer. So when $n\geq 3$, the image of $[0,\delta]$ is $\min\{ n\delta,1\} > 2\delta = \mu_\T(A+A)-\mu_\T(A)$. It remains to verify $n=2$, and the shortest compact arc whose preimage contains $A$ is $[-1/3, 2\delta]$, has size  $1/3+2\delta$. 
\end{example}

\section{A strengthened Kneser's theorem}\label{sec:directional}

In this section we study the doubling of a set when its projection in $\T$ lies in  a short interval. We will prove the following strengthening of Kneser's inequality. The idea is similar to the sumset from subset result by the first author and Mudgal~\cite{JM} that if $A,B\subseteq\T$ and $B$ is contained in a small interval $I$, then $A+B$ can be approximated by $A$ translated by the two end points and a random point of $I$. Here we will show that in this situation one can also improve the lower bound on $\mu(A+B)$, and the proof is simpler as we only use two end points to estimate the sumset. 

\begin{proposition}\label{prop: strong kneser}
    Let $G$ be a connected compact abelian group and $\rho: G\to\T$ be a continuous surjective group homomorphism. 
Suppose $A,B\subseteq G$ are compact sets and the shortest intervals containing $\rho(A)$ and $\rho(B)$ have length $h_A$ and $h_B$ respectively. Suppose $h_A+h_B<1$. If $\mu_G(A)+\mu_G(B)>\max\{h_A,h_B\}$, we have
     \[
     \mu_G(A+B)\geq \max\{ \mu_G(A)+h_B, \mu_G(B)+h_A\}
     \]
\end{proposition}
\begin{proof}
    Let $I_A,I_B\subseteq \T$ be the shortest intervals containing $\rho(A)$ and $\rho(B)$, and $\mu_\T(I_A)=h_A$, $\mu_\T(I_B)=h_B$. By compactness we may assume $\rho(A)$ and $\rho(B)$ contain the boundary points of $I_A$ and $I_B$. By quotient integral formula, we have $\mu_G(A)\leq h_A$ and $\mu_G(B)\leq h_B$. Without loss of generality we also assume that $0\in A\cap B$, and $I_A=[0,h_A]$, $I_B=[0,h_B]$. Let $H=\ker\rho$. Note that $H$ may not be connected. By symmetry, we further assume $h_A\leq h_B$. 

    We now consider the pullback of the universal covering $\R\to\T$, and define
    \[
    \widetilde{G} = \{(g,t)\in G\times \R : \rho(g) = t \pmod{\mathbb Z}\}. 
    \]
    We use $p_1:\widetilde{G} \to G$ and $p_2:\widetilde{G}\to\R$ to denote the coordinate projections. Then $\ker(p_1)\cong \Z$ and $\ker (p_2)\cong H$. Define the lifts of $A$ and $B$ by identifying $I_A$ and $I_B$ intervals in $\R$
    \[
    \widetilde{A} = p_1^{-1}(A) \cap p_2^{-1}(I_A) \quad\text{and}\quad \widetilde{B} = p_1^{-1}(B) \cap p_2^{-1}(I_B). 
    \]
    Choose a Haar measure $\mu_{\widetilde{G}}$ on $\widetilde{G}$ via the quotient integral formula along the exact sequence
    \[
    0\to H \to \widetilde{G} \xrightarrow{p_2} \R \to 0
    \]
    by choosing the normalized measure on $H$ and Lebesgue measure on $\R$. Therefore $\mu_{\widetilde{G}}(\widetilde{A})=\mu_G(A)$ and $\mu_{\widetilde{G}}(\widetilde{B})=\mu_G(B)$. Also, since $p_2(\widetilde{A}+\widetilde{B})\subseteq [0,h_A+h_B]$ and $h_A+h_B<1$, the map $p_1$ is injective on $\widetilde{A}+\widetilde{B}$, which implies
    \[
    \mu_{\widetilde{G}}(\widetilde{A}+\widetilde{B}) = \mu_G(A+B). 
    \]

    Let $b\in \widetilde{B}$ such that $b\in p_2^{-1}(h_B)$. Now we quotient out the discrete subgroup $q:\widetilde{G}\to L=\widetilde{G}/b\mathbb Z$. Clearly, $L$ is compact. Let $\nu$ be the Haar measure on $L$ given by the quotient formula, where we use the normalzed $\mu_{\widetilde{G}}$ on $\widetilde{G}$ and the counting measure on the kernel, that is, for any compact set $\Omega\subseteq \widetilde{G}$
   \begin{equation}\label{eq: qif for L}
\mu_{\widetilde{G}}(\Omega) = \int_{L} |\Omega \cap q^{-1}(h)| \dd \nu(h). 
   \   \end{equation}
Note that $|\Omega \cap q^{-1}(h)|<\infty$ when $D$ is compact. We also have $\nu(L)=h_B$ from the construction. 
 
 On the other hand $q$ and $p_2$ induce a short exact sequence
    \begin{equation}\label{eq: L}
    0\to H\to L\xrightarrow{p_2'} \R/h_B\Z\to0. 
    \end{equation}
 The induce map $p_2'$ pushes $\nu$ on $L$ to a space of Lebesgue measure $h_B$, and hence $\nu(p_2'^{-1}(0))=0$. As the map $q$ is injective on $\widetilde{B}$ except for the end points and end points lie in a $\nu$-null set wo conclude that $\nu(q(\widetilde{B}))=\mu_G(B)$. 

 Observe that the lattice counting function satisfies the pointwise estimate 
 \[
|(\widetilde{A} + \widetilde{B}) \cap q^{-1}(h)| \geq |\widetilde{A}\cap q^{-1}(h)| + \one_{q(\widetilde{A})+q(\widetilde{B})}(h). 
 \]
Indeed, the above estimate is equivalent to $\widetilde{A}+\widetilde{B}\supseteq \widetilde{A}\cup (\widetilde{A}+b)$. On a fiber meeting $\widetilde{A}$, the union of a nonempty finite subset of a translate of $b\Z$ with its translate by $b$ has at least one more point. Using the quotient integral formula \eqref{eq: qif for L},
\begin{equation}\label{eq: bound assymetric}
\mu_G(A+B)=\mu_{\widetilde{G}}(\widetilde{A}+\widetilde{B})\geq \mu_{\widetilde{G}}(\widetilde{A}) + \nu(q(\widetilde{A})+q(\widetilde{B}))=\mu_G(A) + \nu(q(\widetilde{A})+q(\widetilde{B})). 
\end{equation}

As $h_A\leq h_B$, we have $\nu(q(\widetilde{A}))=\mu_G(A)$, and hence
\[
\nu(q(\widetilde{A})) + \nu(q(\widetilde{B}))\geq \mu_G(A) + \mu_G(B) \geq h_B. 
\]
This implies that $q(\widetilde{A})+q(\widetilde{B}) = L$. Therefore by \eqref{eq: bound assymetric} we have
\[
\mu_G(A+B) \geq \mu_G(A) + h_B. 
\]
We also conclude $\mu_G(A+B)\geq \mu_G(B) + h_A$ by switching the role of $A$ and $B$. The only thing to note is that, as $h_A\leq h_B$ we no longer have, after switching the role, $\nu(q(\widetilde{B}))=\mu_G(B)$, and instead, $\nu(q(\widetilde{B}))=\mu_G(B) - (h_B-h_A)$. But this still implies $q(\widetilde{A})+q(\widetilde{B}) = L$ and the conclusion follows. 
\end{proof}

We remark that a large part of the argument is arguing the measure between maps, as we do not have $\ker(\rho)$ is connected. With the connectedness assumption the argument can be shortened. Proposition~\ref{prop: strong kneser} implies the following version of $3k-4$ type result.

\begin{corollary}\label{cor: restricted 3k-4}
Let $G$ be a connected compact abelian group and $\rho: G\to\T$ be a continuous surjective group homomorphism with connected kernel. 
Suppose $A\subseteq G$ is a compact set and the shortest interval containing $\rho(A)$ has length $h_A$. Suppose $h_A<1/2$. Then
     \[
     \mu_G(A+A)\geq \min\{ \mu_G(A)+h_A, 3\mu_G(A)\}. 
     \]    
\end{corollary}
\begin{proof}
Note that by Lemma~\ref{lem: connected}, connectedness on $\ker \rho$ forces $L$ in \eqref{eq: L} to be connected. 
     Apply Kneser's theorem to $L$ in \eqref{eq: bound assymetric}, we have
    \[
    \mu_G(A+A)\geq \mu_G(A) + \min\{2\mu_G(A), h_A\}
    \]
    as desired. 
\end{proof}

For the readers who cares about the projection theorem, Proposition~\ref{prop: strong kneser} implies the following result. 

\begin{corollary}
    Let $0\to H\to G\xrightarrow{\pi} Q\to 0$ be an exact sequence of connected compact abelian groups, and $A\subseteq G$ be compact. We further assume there is $\rho: Q\to\T$ a continuous surjectuve group homomorphism with connected kernal such that the shortest interval containing $\rho(\pi(A))$ is at most $1/2$. If $\mu_G(A+A)<3\mu_G(A)$, then
    \[
    \mu_{Q}(\pi(A)+\pi(A)) \leq 2(\mu_G(A+A)-\mu_G(A)). 
    \]
\end{corollary}
\begin{proof}
The kernel of $\rho\circ\pi$ is an extension of $\ker\rho$ by $H$, and is connected.  By Corollary~\ref{cor: restricted 3k-4}, write $h_A$ the length of the shortest compact interval $I_A$ containing $\rho(\pi(A))$, we have $h_A\leq \mu_G(A+A)-\mu_G(A)$. Observe that $\pi(A)+\pi(A)\subseteq I_A+I_A $, hence
\[
\mu_Q(\pi(A)+\pi(A)) \leq 2h_A \leq 2(\mu_G(A+A)-\mu_G(A))
\]
as claimed. 
\end{proof}

\section{Proof of Theorem~\ref{thm: main Freiman}}

We will prove the Freiman theorem following projecting and lifting argument in \cite[Section 11]{JingTranKIP}. We will use the compact connected abelian case of the Hilbert 5th problem by Pontrjagin~\cite{Pontrjagin}. The general statement which also works for nonabelian groups is known as the Gleason--Yamabe theorem~\cite{Gleason,Yamabe}. 
\begin{fact}\label{fact: abelian GY}
    Let $G$ be a compact connected abelian group. For every neighbourhood $U$ of $0$, there is a compact subgroup $H\subseteq U$ such that $G/H$ is a finite dimensional torus. 
\end{fact}

Note that the kernel $H$ needs not to be connected.

\begin{proof}[Proof of Theorem~\ref{thm: main Freiman}]
Let $\theta_{2\tau}$ be the constant in Fact~\ref{fact:toral} for  parameter $2\tau$, and let $\eta = \min\{\theta_{2\tau}/2,1/4\}$. 
  Choose
\begin{equation*}
 0<\varepsilon<\min\{\mu_G(A),\,\mu_G(A)+2\mu_G(B)-\mu_G(A+B)\}.
\end{equation*}

For each of the sets $A,B,A+B$, the outer regularity of Haar measure guarantees open sets $U_A$, $U_B$, $U_{A+B}$ that approximate them from outside with measure difference at most $\varepsilon$.
By compactness of these three sets, there is a common neighbourhood $U$ of $0$ such that $X+U\subseteq U_X$ for all three choices of $X=A,B,A+B$. Now we apply the structure theorem of compact abelian groups (Fact~\ref{fact: abelian GY}) and obtain $H\subseteq U$ with $G/H \cong \T^d$ for some $d<\infty$. Let $\pi: G\to G/H$ be the natural projection. Therefore, for $X=A,B,A+B$, using quotient integral formula one has
\[
\mu_G(X)\leq \mu_{\T^d}(\pi(X)) = \mu_G(X+K)\leq \mu_G(U_X)<\mu_G(X)+\varepsilon. 
\]

We next verify the image satisfies the small doubling condition. This is straightforward, as 
\begin{align*}
\mu_{\T^d}(\pi(A)+\pi(B)) &< \mu_G(A+B)+\varepsilon <\mu_G(A)+2\mu_G(B)\\ 
&\leq \mu_{\T^d}(\pi(A))+\mu_{\T^d}(\pi(B))+\min\{\mu_{\T^d}(\pi(A)),\mu_{\T^d}(\pi(B))\}. 
\end{align*}
Also the ratio of $\mu_{\T^d}(\pi(A))$ and $\mu_{\T^d}(\pi(B))$ is controlled, as
\[
\max\{\mu_{\T^d}(\pi(A)),\mu_{\T^d}(\pi(B))\} <\mu_G(A)+\varepsilon <2\mu_G(A)\leq \theta_{2\tau},
\]
and
\[
\min\{\mu_{\T^d}(\pi(A)),\mu_{\T^d}(\pi(B))\} \geq \mu_G(B)\geq \frac{\mu_G(A)}{\tau}>\frac{\max\{\mu_{\T^d}(\pi(A)),\mu_{\T^d}(\pi(B))\}}{2\tau}.
\]
We now apply Bilu's theorem (Fact~\ref{fact:toral}), so there is a continuous surjective character $\chi:\T^d\to\T$, and two compact intervals $I_A,I_B\subseteq \T$ with $\pi(A)\subseteq\chi^{-1}(I_A)$, $\pi(A)\subseteq\chi^{-1}(I_B)$, and
\[
\mu_\T(I_A)\leq \mu_{\T^d}(\pi(A)+\pi(B))-\mu_{\T^d}(\pi(B)),\quad \mu_\T(I_B)\leq \mu_{\T^d}(\pi(A)+\pi(B))-\mu_{\T^d}(\pi(A)). 
\]

Define the continuous surjective character $\psi=\chi\circ \pi: G\to\T$. Then clearly $A\subseteq \psi^{-1}(I_A)$ and $B\subseteq \psi^{-1}(I_B)$. Moreover, 
\[
\mu_\T(I_A)\leq \mu_{G}(A+B)-\mu_{G}(B)+\varepsilon<\mu_G(A)+\mu_G(B), 
\]
and similarly 
\[
\mu_\T(I_B)\leq \mu_{G}(A+B)-\mu_{G}(A)+\varepsilon<2\mu_G(B)\leq \mu_G(A)+\mu_G(B). 
\]
Suppose $I,J\subseteq\T$ are the shortest compact intervals containing $\psi(A)$ and $\psi(B)$ respectively, then $I\subseteq I_A$ and $J\subseteq I_B$, and with $h_A=\mu_\T(I)$ and $h_B=\mu_\T(J)$, 
\[
\max\{h_A,h_B\}<\mu_G(A)+\mu_G(B),\quad\text{and}\quad h_A+h_B\leq 4\mu_G(A)<1. 
\]
Hence we may apply Proposition~\ref{prop: strong kneser} and conclude
\[
\mu_G(A+B)\geq \max\{\mu_G(A)+h_B, \mu_G(B)+h_A\}. 
\]
This can be rewritten as
\[
\mu_\T(I)\leq \mu_G(A+B)-\mu_G(B),\quad\text{and}\quad \mu_\T(J)\leq \mu_G(A+B)-\mu_G(A).
\]
Together with the fact that $A\subseteq \psi^{-1}(I)$ and $B\subseteq\psi^{-1}(J)$ finishing the proof. 
\end{proof}

We remark that Proposition~\ref{prop: strong kneser} allow disconnected kernels, which is essential in the argument.  For the universal solenoid $G=\widehat{\mathbb Q}$, every nonzero character corresponds to $q\in\mathbb Q$ and its kernel has dual $\mathbb Q/q\Z$, which has torsion. Thus every such kernel is disconnected. Neither the Gleason--Yamabe kernel nor the final character kernel is required to be connected in our proof.

\section{Proof of the asymmetric projection theorem}
\label{sec:projection-proof}

\begin{theorem}
\label{thm:projection}
For every $\tau>0$ there is a constant $\eta_\tau>0$ with the following property.  Let $0\to H\to G\xrightarrow{\pi} Q\to 0$ be an exact sequence of compact connected abelian groups. Suppose $A,B\subseteq G$ are compact and
\[
 0<\tau^{-1}\mu_G(A)\leq\mu_G(B)\leq\mu_G(A)\leq\eta_\tau.
\]
Suppose
\[
\mu_G(A+B)<\mu_G(A)+2\mu_G(B). 
\]
Then we have either $\pi(A)+\pi(B)=Q$, or
\[
\mu_Q(\pi(A)+\pi(B))\leq 2\mu_G(A+B)-\mu_G(A)-\mu_G(B). 
\]
\end{theorem}

We now prove Theorem~\ref{thm:projection}.  For every $\tau\geq1$, we will choose
\[
 0<\varepsilon_\tau<\min\{\eta_\tau,1/4\},
\]
where $\eta_\tau$ is the constant in Theorem~\ref{thm: main Freiman}.

\begin{proof}[Proof of Theorem~\ref{thm:projection}]
By the assumption we have
\[
 0<\mu_G(A)\leq\mu_Q(\pi(A))<\varepsilon_\tau\leq\eta_\tau.
\]
We then apply Theorem~\ref{thm: main Freiman} to $A,B\subseteq G$, which gives a surjective character $\chi:G\to\T$ and compact arcs $I,J\subseteq\T$ such that $A\subseteq\chi^{-1}(I)$, $B\subseteq\chi^{-1}(J)$, and 
\[
\mu_\T(I)\leq \mu_G(A+B)-\mu_G(B),\qquad \mu_\T(J)\leq \mu_G(A+B)-\mu_G(A). 
\]

We claim that $\chi|_H = 0$. As otherwise, the connectedness of $H$ gives $\chi(H)=\T$. On the coset $g+H$, as $A\subseteq\chi^{-1}(I)$, 
\[
(A-g)\cap H \subseteq (\chi|_H)^{-1}(I - \chi(g)).
\]
Hence we have
\[
\int_H \one_A(g + h)\dd\mu_H(h)\leq \mu_\T(I). 
\]
Hence by quotient integral formula,
\[
\mu_G(A)\leq \mu_Q(\pi(A))\mu_\T(I)<2\mu_Q(\pi(A))\mu_G(A)<\mu_G(A).
\]

Therefore, $H\leq\ker\chi$ and $\chi$ descends to a character $\overline{\chi}:Q\to\T$. It is surjective as $\chi$ is surjective. Moreover, $\pi(A)\subseteq \overline{\chi}^{-1}(I)$ and $\pi(B)\subseteq \overline{\chi}^{-1}(J)$. This implies that $\pi(A)+\pi(B)\subseteq \overline{\chi}^{-1}(I+J)$. By quotient integral formula the pushforward of $\mu_Q$ under $\overline{\chi}$ is the normalized Haar measure, then
\[
\mu_Q(\pi(A)+\pi(B))\leq \mu_\T(I+J)\leq 2\mu_G(A+B)-\mu_G(A)-\mu_G(B)
\]
as claimed.

When $A=B$ this implies $\mu_Q(\pi(A)+\pi(A))\leq 2(\mu_G(A+A)-\mu_G(A))\leq 2(K-1)\mu_G(A)\leq 2(K-1)\mu_Q(\pi(A))$. 
\end{proof}

The following proposition shows that Theorem~\ref{thm:projection} is sharp for all $2\leq K<3$. The construction is essentially an interval union a Cantor set. 

\begin{proposition}\label{prop: sharp}
For every $2\leq K<3$ and every $\varepsilon>0$, there exist an exact
sequence
\[
 0\to\T\xrightarrow{\ell}\T^2
 \xrightarrow{\pi}\T\to0
\]
with $\ell(t)=(0,t)$ and $\pi(x,y)=x$, and a compact set $A\subseteq\T^2$ such that $0<\mu_{\T^2}(A) = \mu_\T(\pi(A))<\varepsilon$, $\mu_{\T^2}(A+A)=K\mu_{\T^2}(A)$, and
\[
 \mu_\T(\pi(A)+\pi(A)) =(2K-2)\mu_\T(\pi(A))=2\bigl(\mu_{\T^2}(2A)-\mu_{\T^2}(A)\bigr).
\]
\end{proposition}

\begin{proof}
Let $0<h<\varepsilon$. Consider a non-wrapping lift of $\T$ to $\R$ and let $\mathcal{C}\subseteq[0,1]$ be the middle $1/3$ Cantor set. Define $I=[0,h]$ and $C=(K-1)h \mathcal{C}\subseteq [0,(K-1)h]$ subsets of $\R$, and $P=I\cup C$. Now we project $I,C,P$ to $\T$. As $C$ is a set of measure $0$, we have $\mu_\T(P)=h$. As gaps in $\mathcal C$ of length at most $(K-1)h/3<h$, we have
\[
[0, (K-1)h+h]\subseteq\bigcup_{x\in C} x+I.
\]
As $C+C=[0,2(K-1)h]$, we have 
\[
I+C=[0,h+(K-1)h],\qquad I+P = [0, h+(K-1)h],\qquad P+P=[0,2(K-1)h]. 
\]
Now we define
\[
A = (I\times \T) \cup (P\times \{0\}). 
\]
Then $\pi(A)=P$, and 
\[
\mu_{\T^2}(A) = h = \mu_\T(\pi(A)). 
\]
Moreover, $A+A = ((I+P)\times\T)\cup((P+P)\times\{0\})$. Then $\mu_{\T^2}(A+A)=h+(K-1)h = K \mu_{\T^2}(A)$, and
\[
\mu_\T(\pi(A)+\pi(A)) = 2(K-1)h = 2(\mu_{\T^2}(A+A)-\mu_{\T^2}(A))
\]
as claimed. 
\end{proof}

We would like to remark that for $5/2<K<3$, the doubling in the projection given by Theorem~\ref{thm:projection} is $2(K-1)>3$. Therefore the $3k-4$ theorem cannot be invoked on $\pi(A)$. Also the fiber $C\times\{0\}$ in the construction is a measure $0$ set, but it is essential for the doubling and the projection, hence any argument which discards zero measure fibers cannot prove an exact containment theorem.

We finish the section by recording another example showing that connectedness of the kernel in Theorem~\ref{thm:projection} is necessary.

\begin{proposition}\label{prop: finite cover}
For every integer $n\geq3$ and every $\varepsilon>0$, there exist a compact set $A\subseteq\T$ and the covering homomorphism $\pi_n:\T\to\T$ where $\pi_n(x)=nx$, such that $0<\mu_\T(\pi_n(A))<\varepsilon$, and $2<\mu_\T(A+A)/\mu_\T(A)<3$, but
\[
 \mu_\T(\pi_n(A)+\pi_n(A)) >2\big(\mu_\T(A+A)-\mu_\T(A)\big).
\]
\end{proposition}

\begin{proof}
We define
\[
A = [0,a]\cup \{p\} \subseteq\T.
\]
Here $a<p<2a$, and $np<1/2$. We further assume $na<\varepsilon$. Hence there is no wrapping around and
\[
A+A = [0,a+p]\cup \{2p\}.
\]
Therefore $\mu_\T(A+A)= a+p = (1+p/a)\mu_\T(A)$.

On the other hand, 
\[
 \pi_n(A)=[0,na]\cup\{np\},
 \qquad
 \pi_n(A)+\pi_n(A) =[0,n(a+p)]\cup\{2np\},
\]
and hence $ \mu_\T(\pi_n(A))=na$, and $\mu_\T(\pi_n(A)+\pi_n(A))=n(a+p)$. 
A direct calculation gives the conclusion. 
\end{proof}

\section{Popular sumsets}

In this section we consider variations of Theorem~\ref{thm: main Freiman}. The extra input is the following version of arithmetic removal lemma by Candela--Szegedy--Vena~\cite[Theorem~1.3]{CSV} for locally compact abelian groups.

\begin{fact}\label{fact:pop removal}
There is a function $\mathcal R:(0,1)\to(0,1]$ with the following property.  Let $0<r<1$, let $G$ be a compact Hausdorff abelian group, and let $U,V,W\subseteq G$ be Borel sets.  If
\begin{equation*}
 \int_{G^2}\one_U(x)\one_V(y)\one_W(x+y) \dd\mu_G(x)\dd\mu_G(y)\leq\mathcal R(r),
\end{equation*}
then there are Borel sets
\[
 R_U\subseteq U,\qquad R_V\subseteq V,\qquad R_W\subseteq W,
\]
with $\mu_G(R_U),\mu_G(R_V),\mu_G(R_W)\leq r$ such that
\begin{equation*}
 \bigl((U\setminus R_U)+(V\setminus R_V)\bigr)
       \cap(W\setminus R_W)=\varnothing.
\end{equation*}
\end{fact}

\begin{lemma}\label{lem: pop compact subset}
Let $G$ be a compact abelian group, let $A,B\subseteq G$ be compact, and let $S\subseteq G$ be Borel. Suppose $0<r<1$ and
\begin{equation*}
 \int_{G\setminus S}\one_A*\one_B(z)\dd\mu_G(z)\leq\mathcal R(r).
\end{equation*}
Then there are compact sets $C\subseteq A$ and $D\subseteq B$ such that
\begin{equation*}
 \mu_G(A\setminus C)\leq2r,\qquad
 \mu_G(B\setminus D)\leq2r,\qquad
 \mu_G(C+D)\leq\mu_G(S)+r.
\end{equation*}
More precisely, there is a Borel set $R\subseteq G\setminus S$
with $\mu_G(R)\leq r$ such that $C+D\subseteq S\cup R$.
\end{lemma}

\begin{proof}
Using Fubini's theorem and translation invariance we get
\begin{align*}
 &\int_{G^2}\one_A(x)\one_B(y)\one_{G\setminus S}(x+y)
                 \dd\mu_G(x)\dd\mu_G(y)\\
 &=\int_{G\setminus S}\int_G\one_A(x)\one_B(z-x)
                 \dd\mu_G(x)\dd\mu_G(z)
 =\int_{G\setminus S} \one_A*\one_B(z)\dd\mu_G(z)
 \leq\mathcal R(r).
\end{align*}
Apply Fact~\ref{fact:pop removal} to $A,B,G\setminus S$.  It gives Borel sets $R_A\subseteq A$, $R_B\subseteq B$, and $R\subseteq G\setminus S$, each of measure at most $r$, such that
\[
 \bigl((A\setminus R_A)+(B\setminus R_B)\bigr)
       \cap((G\setminus S)\setminus R)=\varnothing.
\]
By inner regularity, we may choose compact subsets $C\subseteq A\setminus R_A$ and $D\subseteq B\setminus R_B$ such that $\mu_G((A\setminus R_A)\setminus C)<r$ and $ \mu_G((B\setminus R_B)\setminus D)<r$
Consequently, $\mu_G(A\setminus C)\leq2r$ and $\mu_G(B\setminus D)\leq2r$. We also have
\[
 C+D\subseteq G\setminus((G\setminus S)\setminus R)=S\cup R.
\]
The sum $C+D$ is compact, so it is measurable.  This proves the desired bound. 
\end{proof}

\begin{proposition}\label{prop: pop transfer}
Let $\tau\geq1$, let $G$ be compact connected abelian, and let $A,B\subseteq G$ be compact. Suppose $0<r\leq \mu_G(B)/4$, and
\[
0<\tau^{-1}\mu_G(A)\leq \mu_G(B)\leq \mu_G(A)\leq\eta_\tau.
\]
Suppose also that $0<t\leq\mathcal R(r)$, and
\[
S_t(A,B)<\mu_G(A)+2\mu_G(B) - 7r
\]
Then there are compact sets $C\subseteq A$ and $D\subseteq B$ with $\mu_G(A\setminus C)\leq2r$ and $\mu_G(B\setminus D)\leq2r$ such that
\[
\mu_G(C+D)\leq S_t(A,B)+r<\mu_G(C)+\mu_G(D)+\min\{\mu_G(C),\mu_G(D)\},
\]
together with a continuous surjective homomorphism $\chi:G\to\T$, and compact arcs $I,J\subseteq\T$ with $C\subseteq\chi^{-1}(I)$, $D\subseteq\chi^{-1}(J)$, and
\[
\mu_G\bigl(A\setminus\chi^{-1}(I)\bigr)\leq2r,\qquad
   \mu_G\bigl(B\setminus\chi^{-1}(J)\bigr)\leq2r,
\]
and $\mu_\T(I)\leq S_t(A,B)-\mu_G(B)+3r$, $\mu_\T(J)\leq S_t(A,B)-\mu_G(A)+3r$.
\end{proposition}

\begin{proof}
Apply Lemma~\ref{lem: pop compact subset} with $S=S_t(A,B)$. The resulting compact sets satisfy $\mu_G(A)-2r\leq \mu_G(C)\leq \mu_G(A)$, $\mu_G(B)-2r\leq \mu_G(D)\leq\mu_G(B)$, and
\[
 \mu_G(C+D)\leq S_t(A,B)+r.
\]
Note that the original order $\mu_G(B)\leq \mu_G(A)$ need not imply the order for $\mu_G(C)$ and $\mu_G(D)$. Let us now verify the hypothesis of Theorem~\ref{thm: main Freiman} for sets $C$ and $D$. 

As $r\leq \mu_G(B)/4$ and $\mu_G(B)\geq \mu_G(A)/\tau$,
\[
 0<\frac{1}{2\tau}\max\{\mu_G(C),\mu_G(D)\} \leq\frac{\mu_G(A)}{2\tau}\leq\frac{\mu_G(B)}{2} \leq \mu_G(B)-2r
 \leq\min\{\mu_G(C),\mu_G(D)\}.
\]
Also $\max\{\mu_G(C),\mu_G(D)\}\leq \mu_G(A)\leq\eta_\tau^{\rm pop}\leq\theta_{2\tau}$. For the growth hypothesis, observe that
\begin{align*}
 \mu_G(C)+\mu_G(D)+\min\{\mu_G(C),\mu_G(D)\}
 &\geq(\mu_G(A)-2r)+(\mu_G(B)-2r)+(\mu_G(B)-2r)\\
 &=\mu_G(A)+2\mu_G(B)-6r.
\end{align*}
Using this, we have
\[
 \mu_G(C+D)\leq S_t(A,B)+r< \mu_G(C)+\mu_G(D)+\min\{\mu_G(C),\mu_G(D)\}.
\]
Now we may apply Theorem~\ref{thm: main Freiman} with parameter $2\tau$. Then we obtain a continuous surjective $\chi:G\to\T$ and arcs $I,J\subseteq \T$ with $C\subseteq\chi^{-1}(I)$, $D\subseteq\chi^{-1}(J)$, and 
\[
\mu_\T(I)\leq\mu_G(C+D)-\mu_G(D), \qquad \mu_\T(J)\leq\mu_G(C+D)-\mu_G(C).
\]
Together with the above estimation, we conclude that
\[
 \mu_\T(I)\leq S_t(A,B)-\mu_\T(B)+3r,\qquad
 \mu_\T(J)\leq S_t(A,B)-\mu_\T(A)+3r.
\]
Finally, the exceptional sets are contained in $A\setminus C$ and $B\setminus D$, respectively. This proves the proposition.
\end{proof}

Now we are ready to prove Theorem~\ref{thm: popular Freiman}. 

\begin{proof}[Proof of Theorem~\ref{thm: popular Freiman}]
Choose $\eta_\tau$ as in Proposition~\ref{prop: finite cover}, and let $t_0 = \mathcal R(r)$, and
\[
 r=\frac{\alpha}{20\tau}\min\{\kappa,\varepsilon\}. 
\]
Note that all parameters are positive and independent of $G,A,B$. Since $\mu_G(B)\geq \mu_G(A)/\tau\geq\alpha/\tau$,
\[
 r\leq\frac {\mu_G(B)}{20}\min\{\kappa,\varepsilon\}<\frac {\mu_G(B)}{4},
 \qquad 7r\leq\frac7{20}\kappa \mu_G(B)<\kappa \mu_G(B).
\]
Consequently, the popular sumset hypothesis implies
\[
 S_t(A,B)+7r\leq \mu_G(A)+(2-\kappa)\mu_G(B)+7r<\mu_G(A)+2\mu_G(B).
\]
Proposition~\ref{prop: pop transfer} applies and gives a common character $\chi$ and arcs $I,J\subseteq\T$. Its exceptional measures are bounded by
\[
 2r\leq\frac{\varepsilon \mu_G(B)}{10}\leq\varepsilon \mu_G(B),
\]
and its two additional width errors can be bounded by
\[
 3r\leq\frac{3\varepsilon \mu_G(B)}{20}\leq\varepsilon \mu_G(B).
\]
Hence this finishes the proof. 
\end{proof}

Finally we prove Theorem~\ref{thm: pop truncated}.

\begin{proof}
Observe that pointwise $t\one_{S_t(A,B)}\leq\min\{\one_A*\one_B,t\}$.  Integration this inequality gives
\[
S_t(A,B)\leq \frac{1}{t}\int_G \min\{\one_A*\one_B(x),t\} \dd\mu_G(x).
\]
Thus Theorem~\ref{thm: popular Freiman} applies and the conclusion follows. 
\end{proof}

\bibliographystyle{amsalpha}
\bibliography{reference}

\providecommand{\bysame}{\leavevmode\hbox to3em{\hrulefill}\thinspace}
\providecommand{\MR}{\relax\ifhmode\unskip\space\fi MR }
% \MRhref is called by the amsart/book/proc definition of \MR.
\providecommand{\MRhref}[2]{%
  \href{http://www.ams.org/mathscinet-getitem?mr=#1}{#2}
}
\providecommand{\href}[2]{#2}
\begin{thebibliography}{CDR19}

\bibitem[Bil98]{Bilu}
Yuri Bilu, \emph{The {$(\alpha+2\beta)$}-inequality on a torus}, J. London
  Math. Soc. (2) \textbf{57} (1998), no.~3, 513--528. \MR{1659821}

\bibitem[CDR19]{CandelaDeRoton}
Pablo Candela and Anne De~Roton, \emph{On sets with small sumset in the
  circle}, Q. J. Math. \textbf{70} (2019), no.~1, 49--69. \MR{3927843}

\bibitem[CI22]{CI22}
Michael Christ and Marina Iliopoulou, \emph{Inequalities of {R}iesz-{S}obolev
  type for compact connected {A}belian groups}, Amer. J. Math. \textbf{144}
  (2022), no.~5, 1367--1435. \MR{4494185}

\bibitem[CSV16]{CSV}
Pablo Candela, Bal\'azs Szegedy, and Llu\'is Vena, \emph{On linear
  configurations in subsets of compact abelian groups, and invariant measurable
  hypergraphs}, Ann. Comb. \textbf{20} (2016), no.~3, 487--524. \MR{3537915}

\bibitem[Fol95]{Folland}
Gerald~B. Folland, \emph{A course in abstract harmonic analysis}, Studies in
  Advanced Mathematics, CRC Press, Boca Raton, FL, 1995. \MR{1397028}

\bibitem[Gle52]{Gleason}
Andrew~M. Gleason, \emph{Groups without small subgroups}, Ann. of Math. (2)
  \textbf{56} (1952), 193--212. \MR{49203}

\bibitem[Gri19]{Griesmer}
John~T. Griesmer, \emph{Semicontinuity of structure for small sumsets in
  compact abelian groups}, Discrete Anal. (2019), Paper No. 18, 46.
  \MR{4042161}

\bibitem[Jin21]{J}
Yifan Jing,
  \emph{\zh{紧阿贝尔群上最小扩张集的稳定性定理的新证明}},
  \zh{蛙鸣} \textbf{65} (2021), 106--111.

\bibitem[JM23]{JM}
Yifan Jing and Akshat Mudgal, \emph{Kemperman's inequality and {F}reiman's
  lemma via few translates}, 2023, Preprint, arXiv:2307.03066.

\bibitem[JT23]{JingTranKIP}
Yifan Jing and Chieu-Minh Tran, \emph{Measure growth in compact semisimple
  {L}ie groups and the {K}emperman inverse problem}, 2023, Preprint,
  arXiv:2303.15628.

\bibitem[Kem60]{KempermanKneser}
J.~H.~B. Kemperman, \emph{On small sumsets in an abelian group}, Acta Math.
  \textbf{103} (1960), 63--88. \MR{110747}

\bibitem[Kem64]{Kemperman}
\bysame, \emph{On products of sets in a locally compact group}, Fund. Math.
  \textbf{56} (1964), 51--68. \MR{202913}

\bibitem[Kne56]{Kneser}
Martin Kneser, \emph{Summenmengen in lokalkompakten abelschen {G}ruppen}, Math.
  Z. \textbf{66} (1956), 88--110. \MR{81438}

\bibitem[KPT24]{KPT}
Zuxiang Kong, Fei Peng, and Chieu-Minh Tran, \emph{Measure doubling in
  unimodular locally compact groups and quotients}, 2024, Preprint,
  arXiv:2411.17246.

\bibitem[Pon34]{Pontrjagin}
L.~Pontrjagin, \emph{The theory of topological commutative groups}, Ann. of
  Math. (2) \textbf{35} (1934), no.~2, 361--388. \MR{1503168}

\bibitem[Tao18]{TaoKneser}
Terence Tao, \emph{An inverse theorem for an inequality of {K}neser}, Proc.
  Steklov Inst. Math. \textbf{303} (2018), no.~1, 193--219, Published in
  Russian in Tr. Mat. Inst. Steklova {\bf 303} (2018), 209--238. \MR{3920221}

\bibitem[Yam53]{Yamabe}
Hidehiko Yamabe, \emph{A generalization of a theorem of {G}leason}, Ann. of
  Math. (2) \textbf{58} (1953), 351--365. \MR{58607}

\end{thebibliography}

\end{document}